\documentclass[11pt,twoside,reqno]{amsart}
\usepackage[latin1]{inputenc}
\usepackage[italian,english]{babel}
\usepackage{amssymb,latexsym}
\usepackage{amsmath,amsfonts,amsthm}
\usepackage{paralist}
\usepackage{color}

\numberwithin{equation}{section}

\newcommand{\R}{\mathbb{R}}
\newcommand{\N}{\mathbb{N}}

\newcommand{\s}{\sharp}

\newtheorem{lem}{Lemma}
\newtheorem{thm}{Theorem}

\theoremstyle{remark}
\newtheorem{remark}{Remark}

\begin{document}
\title{zero mass case for a fractional Berestycki-Lions type problem}
\author[Vincenzo Ambrosio]{Vincenzo Ambrosio}
\address{%
Dipartimento di Matematica e Applicazioni\\ Universit\`a degli Studi "Federico II" di Napoli\\
via Cinthia, 80126 Napoli, Italy}
\email{vincenzo.ambrosio2@unina.it}
\keywords{Zero mass case, Fractional Laplacian, Nehari manifold, Orlicz spaces}
\subjclass[2010]{35A15, 35J60, 35R11, 45G05}

\maketitle


\begin{abstract}
In this work we study the following fractional scalar field equation 
\begin{equation*}\label{P}
\left\{
\begin{array}{ll}
(-\Delta)^{s} u = g'(u) \mbox{ in } \R^{N}  \\
u> 0
\end{array}
\right.  
\end{equation*}
where $N\geq 2$, $s\in (0,1)$, $(-\Delta)^{s}$ is the fractional Laplacian and the nonlinearity $g\in C^{2}(\R)$ is such that $g''(0)=0$.
By using variational methods, we prove the existence of a positive solution which is spherically symmetric and decreasing in $r=|x|$.\\

\end{abstract}

\section{Introduction}

\noindent
In this paper we are concerned with the existence of solutions to the following problem
\begin{equation}\label{P}
\left\{
\begin{array}{ll}
(-\Delta)^{s} u = g'(u) \mbox{ in } \R^{N} \\
u> 0
\end{array}
\right.  
\end{equation}
where $N\geq 2$, $s\in (0,1)$ and $g: \R \rightarrow \R$ is a smooth function such that $g''(0)=0$.
Here $(-\Delta)^{s}$ is the fractional Laplacian and it can be defined via Fourier transform by
$$
\mathcal{F}(-\Delta)^{s}f(k)=|k|^{2s} \mathcal{F}f(k)
$$
for $u$ belonging to the Schwartz space $\mathcal{S}(\R^{N})$.

\noindent
Problems like (\ref{P}) are motivated by the study of standing waves solutions $\displaystyle{\psi(x,t)=u(x) e^{-\imath c t}}$ of the fractional Schr\"{o}dinger equation 
$$
\imath \frac{\partial \psi}{\partial t}=(-\Delta)^{s} \psi+ G(x, \psi) \quad \mbox{ in } \R^{N}.
$$
This equation has been proposed by Laskin \cite{Laskin1, Laskin2} as a result of expanding the Feynman path integral, from the Brownian-like to the L\'evy-like quantum mechanical paths. 
After that many papers appeared investigating existence, multiplicity and behavior of solutions to fractional Schr\"{o}dinger equations; see \cite{A, A0, A1, DPPV, FFV, FQT, GS, MBR, Secchi} and references therein.\\
More in general, problems involving fractional operators are receiving a special attention in these last years; indeed 
fractional spaces and nonlocal equations have great applications in many different fields, such as, optimization, finance, continuum mechanics, phase transition phenomena, population dynamics, multiple scattering, minimal surfaces and game theory, as they are the typical outcome of stochastically stabilization of L\'evy processes. The interested reader may consult \cite{DPV, MBRS} and references therein, where a more extensive bibliography and an introduction to the subject are given.
\smallskip

\noindent
In the seminal paper \cite{BL1}, Berestycki and Lions investigated the existence of positive ground state solutions to (\ref{P}) when $s=1$, i.e.
\begin{equation}\label{BLP}
-\Delta u = g'(u) \mbox{ in } \R^{N}.
\end{equation}
Under general assumptions on $g$, they proved that there are no finite energy solutions to (\ref{BLP}) if $g''(0)>0$, while if $g''(0)<0$ or $g''(0)=0$, then it is possible to show the existence of a solution to (\ref{BLP}) via constraint minimization.  The case $g''(0)=0$ is called null mass case and it is related to the Yang-Mills equation; see \cite{G, GNN}.\\ 
Let us note that the case $g''(0)=0$ seems to be more intricate then $g''(0)<0$, since unless $g$ satisfies the condition $|g(u)|\leq c|u|^{\frac{2N}{N-2}}$, the energy functional associated to (\ref{BLP}) may be infinite on a dense set of points in $\mathcal{D}^{1,2}(\R^{N})$ and hence cannot be Fr\'echet-differentiable on $\mathcal{D}^{1,2}(\R^{N})$.

\noindent
The question that naturally arises is whether or not the above classical existence results for the equation (\ref{BLP}) can be extended in the non-local setting. When $g''(0)<0$ (in the case of positive mass), the existence of a ground state has been proved in \cite{CW} by combining the Struwe-Jeanjean monotonicity trick and the Pohozaev identity for the fractional Laplacian. 
Now, our aim is to investigate problem (\ref{P}) when $g''(0)=0$, and $g(u)$ behaves like $|u|^{q}$ for $u$ small and $|u|^{p}$ for $u$ large, with $2<p<\frac{2N}{N-2s}<q$.

\noindent
In order to state our result, we introduce the basic assumptions on the nonlinearity $g$.
Here we will assume that $g: \R \rightarrow \R$ is an odd $C^{2}$-function verifying the following conditions:\\
\begin{compactenum}[(g1)]
\item $0<\mu g(t)\leq g'(t) t\leq g''(t) t^{2}$ for any $t \neq 0$ and for some $\mu>2$;\\
\item $g(0)=g'(0)=g''(0)=0$. There exist $c_{0}, c_{2}, p,q$ with $2<p<2^{*}_{s}:=\frac{2N}{N-2s}<q$ such that
\begin{equation}
\left\{
\begin{array}{ll}
c_{0}|t|^{p}\leq g(t) \mbox{ if } |t|\geq 1\\
c_{0}|t|^{q}\leq g(t) \mbox{ if } |t|\leq 1
\end{array}
\right.  
\end{equation}
and
\begin{equation}
\left\{
\begin{array}{ll}
|g''(t)|\leq c_{2} t^{p-2}  \mbox{ if } |t|\geq 1\\
|g''(t)|\leq c_{2} t^{q-2} \mbox{ if } |t|\leq 1.
\end{array}
\right.  
\end{equation}
\end{compactenum}
\begin{remark}\label{rem1}
The assumptions $g''(t)>0$ for all $t\neq 0$, and $(g2)$ imply the existence of $c_{1}, c_{3}>0$ such that:
\begin{align*}
&c_{0}|t|^{p}\leq g(t)\leq c_{3}|t|^{p} \mbox{ for } |t|\geq 1\\
&c_{0}|t|^{q}\leq g(t)\leq c_{3}|t|^{q} \mbox{ for } |t|\leq 1\\
&|g'(t)|\leq c_{1}|t|^{p-1} \mbox{ for } |t|\geq 1\\
&|g'(t)|\leq c_{1}|t|^{q-1} \mbox{ for } |t|\leq 1.
\end{align*}
\end{remark}
As a model for $g$ we can take the function
\begin{equation*}
g(t)=
\left\{
\begin{array}{ll}
t^{q} &\mbox{ if } t\leq 1 \\
a+b t+c t^{p}  &\mbox{ if } t\geq 1
\end{array}
\right.
\end{equation*}
where $a, b$ and $c$ are constants which make the function $g\in C^{2}$.\\
Let us point out that, when $s=1$, the assumptions $(g1)$ and $(g2)$ have been introduced in \cite{BM} to study positive solutions to a nonlinear field equation set in exterior domain. The authors studied (\ref{BLP}) in the Orlicz space $L^{p}+L^{q}$ which seems to be the natural framework for studying ``zero mass" problems. Subsequently, their approach has been also used in \cite{AP1, BGM, BGM2} to study nonlinear Schr\"odinger equations in $\R^{N}$ with bounded or vanishing potentials. Further results concerning zero mass problems can be found in \cite{ASM, AFM, Struwe}. 
\smallskip

\noindent
The main result of this paper is the following
\begin{thm}\label{thm1}
Let $N\geq 2$, $s\in (0,1)$ and $g$ satisfies $(g1)$ and $(g2)$. Then there exists a positive solution to (\ref{P}) which is spherically symmetric and decreasing in $r=|x|$. 
\end{thm}
\noindent
To deal with problem (\ref{P}), we develop an energy minimization argument on a Nehari manifold. 
\noindent
More precisely, solutions to (\ref{P}) will be obtained by minimizing
$$
I(u)=\frac{1}{2}\iint_{\R^{2N}} \frac{|u(x)-u(y)|^{2}}{|x-y|^{N+2s}} dx dy-\int_{\R^{N}} g(u(x)) dx
$$
on the Nehari manifold
$$
\mathcal{N}=\left\{u\in  \mathcal{D}^{s,2}(\R^{N})\setminus\{0\}: \iint_{\R^{2N}} \frac{|u(x)-u(y)|^{2}}{|x-y|^{N+2s}} dx dy=\int_{\R^{N}} g'(u)u dx \right \}
$$
where 
$$
\mathcal{D}^{s,2}(\R^{N})=\left\{u\in L^{\frac{2N}{N-2s}}(\R^{N}): \iint_{\R^{2N}} \frac{|u(x)-u(y)|^{2}}{|x-y|^{N+2s}} dx dy<\infty \right\}.
$$
In order to obtain the smoothness of the functional $I$, we introduce the Orlicz space $L^{p}+L^{q}$ related to the growth assumptions of $g$ at zero and at infinity.
Then, we show that $I\in C^{2}(\mathcal{D}^{s,2}(\R^{N}), \R)$, and by proving the compactness of the subspace $\tilde{\mathcal{D}}^{s,2}_{rad}(\R^{N})$ of nonnegative radial decreasing functions of $\mathcal{D}^{s,2}(\R^{N})$ in $L^{p}+L^{q}$, we deduce that the infimum of $I$ on $\mathcal{N}$ is achieved at some $u\in \tilde{\mathcal{D}}^{s,2}_{rad}(\R^{N})$.\\
As far as we know the result presented here is new. \\
The paper is organized as follows: in Section $2$ we give some preliminaries about the involved functional spaces, and in Section $3$ we provide the proof of the main result.

\section{preliminaries}
\noindent
In this section we collect some preliminary results which will be useful in the sequel. \\
We denote by $\mathcal{D}^{s,2}(\R^{N})$ the completion of $C^{\infty}_{0}(\R^{N})$ with respect to 
$$
[u]^{2}:=\int_{\R^{N}} |(-\Delta)^{\frac{s}{2}} u|^{2} dx=\int_{\R^{N}} |k|^{2s} |\mathcal{F}f(k)|^{2} dk=\iint_{\R^{2N}} \frac{|u(x)-u(y)|^{2}}{|x-y|^{N+2s}} dx dy.
$$
Then
$$
\mathcal{D}^{s,2}(\R^{N})=\{u\in L^{2^{\s}}(\R^{N}): [u]<\infty \}
$$
where
$
2^{*}_{s}:=\frac{2N}{N-2s}
$
is the fractional Sobolev exponent.\\
Let us denote by $H^{s}(\R^{N})$ the standard fractional Sobolev space, defined as the set of $u\in \mathcal{D}^{s,2}(\R^{N})$ satisfying $u\in L^{2}(\R^{N})$ with the norm 
$$
\|u\|_{H^{s}(\R^{N})}^{2}= [u]^{2} + \|u\|_{L^{2}(\R^{N})}^{2}.   
$$ 
We recall the following embeddings
\begin{thm}\cite{DPV}\label{sobhs}
Let $s\in (0, 1)$ and $N>2s$. There exists a constant $C>0$ such that 
$$
\|u\|_{L^{2^{*}_{s}}(\R^{N})}\leq C[u] \mbox{ for any } u\in \mathcal{D}^{s,2}(\R^{N}).
$$
In particular, $H^{s}(\R^{N})$ is continuously embedded in $L^{q}(\R^{N})$ for any $q\in [2, 2^{*}_{s}]$, and compactly embedded for any $q\in [2, 2^{*}_{s})$.
\end{thm}

\noindent
For more details about fractional Sobolev spaces, we refer to \cite{DPV}. \\
We remark that the symmetric-decreasing rearrangement of a measurable function $u: \R^{N} \rightarrow \R$ that vanishes at infinity (that is $|\{x\in \R^{N}: |u(x)|>a \}|<\infty$ for all $a>0$) is given by
$$
u^{*}(x)=\int_{0}^{\infty} \chi^{*}_{\{|u|>t \}}(x) dt
$$
where $\chi^{*}_{A}=\chi_{A^{*}}$ and $A^{*}=\{x: |x|<r\}$ is such that its volume is that of $A$. For standard properties of rearrangements of functions one can see \cite{LL}.\\
Now, we establish the following fractional Polya-Szeg\"o inequality:
\begin{thm}\label{park}
Let $u:\R^{N} \rightarrow \R$ be a nonnegative measurable function that vanishes at infinity, and let us denote by $u^{*}$ its symmetric-decreasing rearrangement. Assume that $[u]_{H^{s}(\R^{N})}<\infty$. 
Then
\begin{equation}\label{dr}
[u^{*}]_{H^{s}(\R^{N})} \leq  [u]_{H^{s}(\R^{N})}.
\end{equation}
\end{thm}
\begin{proof}
Let 
$$
u_{c}(x)=\min\,  \{\max \, \{u(x)-c,0\},1/c\}
$$
for $c\in (0,1)$. Since $u$ vanishes at infinity, $u_{c}\in L^{2}(\R^{N})$.
In particular $u_{c}\in H^{s}(\R^{N})$ since
$$
|u_{c}(x)-u_{c}(y)| \leq |u(x)-u(y)|
$$
for any $x,y\in \R^{N}$. By monotone convergence Theorem we have 
\begin{equation}\label{mct}
\lim_{c \rightarrow 0} [u_{c}]_{H^{s}(\R^{N})}=[u]_{H^{s}(\R^{N})} \, \mbox{ and } \, \lim_{c \rightarrow 0} [u^{*}_{c}]_{H^{s}(\R^{N})}=[u^{*}]_{H^{s}(\R^{N})}.  
\end{equation}
Now, by using the result in \cite{Park}, we know that 
\begin{equation}\label{Park}
[u^{*}_{c}]_{H^{s}(\R^{N})}\leq [u_{c}]_{H^{s}(\R^{N})}.
\end{equation}
Then taking into account (\ref{mct}) and (\ref{Park}) we deduce the thesis.

\end{proof}

\noindent
We also prove the following useful lemma
\begin{lem}
Let $u\in L^{t}(\R^{N})$, $1\leq t<\infty$ be a nonnegative radial decreasing function (that is $0\leq u(x)\leq u(y)$ if $|x|\geq |y|$). Then 
\begin{equation}\label{CO}
|u(x)|\leq \left(\frac{N}{\omega_{N-1}}\right)^{\frac{1}{t}} |x|^{-\frac{N}{t}} \|u\|_{L^{t}(\R^{N})} \mbox{ for any } x\in \R^{N}\setminus \{0\},
\end{equation}
where $\omega_{N-1}$ is the Lebesgue measure of the unit sphere in $\R^{N}$.
\end{lem} 
\begin{proof}
For all $R>0$ we have, setting $R=|x|$
$$
\|u\|^{t}_{L^{t}(\R^{N})}\geq \omega_{N-1} \int_{0}^{R} (u(r))^{t} r^{N-1} dr\geq \omega_{N-1} (u(R))^{t} \frac{R^{N}}{N}.
$$
\end{proof}

\noindent
Given $p<q$, we define the space $L^{p}+L^{q}:=L^{p}(\R^{N})+L^{q}(\R^{N})$ as the set of functions $u: \R^{N} \rightarrow \R$ such that
$$
u=u_{1}+u_{2}
$$
with $u_{1}\in L^{p}(\R^{N})$ and $u_{2}\in L^{q}(\R^{N})$.
We recall (see \cite{BLof}) that $L^{p}+L^{q}$ is a Banach space with respect to the norm
$$
\|u\|_{L^{p}+L^{q}}=\inf\{ \|u_{1}\|_{L^{p}(\R^{N})}+\|u_{2}\|_{L^{q}(\R^{N})}: u=u_{1}+u_{2} \}. 
$$
Moreover $L^{p}+L^{q}$ coincides with the dual of $L^{p'}\cap L^{q'}$. Then
$$
L^{p}+L^{q}=(L^{p'}\cap L^{q'})^{*}
$$ 
where $p'$ and $q'$ are the conjugate exponent of $p$ and $q$ respectively.
In particular, the norm
$$
\|u\|_{e}=\sup_{w\neq 0}\frac{\displaystyle{\int_{\R^{N}} u(x) v(x) dx}}{\|u\|_{L^{p'}(\R^{N})}+\|v\|_{L^{q'}(\R^{N})}}
$$
is equivalent to $\|\cdot\|_{L^{p}+L^{q}}$.\\
Actually $L^{p}+L^{q}$ is an Orlicz space with $N$-function (see \cite{Adams})
$$
A(u)=\max\{|u|^{p}, |u|^{q}\}.
$$
\noindent

Now we state some useful lemmas whose proofs can be obtained following those in \cite{BF, BM}.
\begin{lem}\label{lem2.1}
\noindent
\begin{compactenum}[(a)]
\item If $u\in L^{p}+L^{q}$, the following inequalities hold:
\begin{align*}
\max \Bigl\{\|u\|_{L^{q}(\R^{N} - \Gamma_{u})}-1 , &\frac{1}{1+ |\Gamma_{u}|^{\frac{1}{r}}} \|u\|_{L^{p}(\Gamma_{u})} \Bigr\} \leq \|u\|_{L^{p}+L^{q}}\\
& \leq \max \{ \|u\|_{L^{q}(\R^{N}-\Gamma_{u})}, \|u\|_{L^{p}(\Gamma_{u})}\}
\end{align*} 
where $r= \frac{pq}{q-p}$ and $\Gamma_{u}=\{x\in \R^{N}: |u(x)|>1\}$.
\item Let $\{u_{j}\}\subset L^{p}+L^{q}$ and set $\Gamma_{j} = \{x\in \R^{N} : |u_{j}(x)|>1\}$. Then $\{u_{j}\}$ is bounded in $L^{p}+L^{q}$ if and only if the sequences $\{|\Gamma_{j}|\}$ and $\{\|u_{j}\|_{L^{q}(\R^{N}- \Gamma_{j})}+ \|u_{j}\|_{L^{p}(\Gamma_{j})}\}$ are bounded. 
\item $g'$ is a bounded map from $L^{p}+L^{q}$ into $L^{\frac{p}{p-1}}\cap L^{\frac{q}{q-1}}$. 
\end{compactenum}
\end{lem}

\begin{lem}\label{lem2.3}
\noindent
\begin{compactenum}[(a)]
\item If $\theta, u$ are bounded in $L^{p}+L^{q}$, then $g''(\theta)u$ is bounded in $L^{p'}\cap L^{q'}$.
\item $g''$ is a bounded map from $L^{p}+L^{q}$ into $L^{\frac{p}{p-2}}\cap L^{\frac{q}{q-2}}$. 
\item $g''$ is a continuous map from $L^{p}+L^{q}$ into $L^{\frac{p}{p-2}}\cap L^{\frac{q}{q-2}}$. 
\item The map $(u, v)\mapsto uv$ from $(L^{p}+L^{q})^{2}$ in $L^{\frac{p}{2}}+L^{\frac{q}{2}}$ is bounded.  
\end{compactenum}
\end{lem}

\begin{lem}\label{lem2.4}
The functional $H: L^{p}+L^{q}\rightarrow \R$ defined by
$$
H(u)=\int_{\R^{N}} g(u(x)) dx
$$
is of class $C^{2}$. 
\end{lem}

\begin{lem}\label{lem2.5}
If the sequence $\{u_{j}\}$ converges to $u$ in $L^{p}+L^{q}$, then the sequence $\{\int_{\R^{N}} g'(u_{j})u_{j} \,dx\}$ converges to $\int_{\R^{N}} g'(u)u \,dx$.
\end{lem}

\begin{remark}
By Lemma \ref{lem2.1} $(a)$ we have $L^{2^{*}_{s}}(\R^{N})\subset L^{p}+L^{q}$ when $2<p<2^{*}_{s}<q$. 
In fact, by using $p<2^{*}_{s}<q$ we find for any $u\in L^{2^{*}_{s}}(\R^{N})$
$$
\|u\|^{q}_{L^{q}(\R^{N}-\Gamma_{u})}\leq \|u\|^{2^{*}_{s}}_{L^{2^{*}_{s}}(\R^{N}-\Gamma_{u})}
$$
and
$$
\|u\|^{p}_{L^{p}(\Gamma_{u})}\leq \|u\|^{2^{*}_{s}}_{L^{2^{*}_{s}}(\Gamma_{u})}
$$
which together with $(a)$ of Lemma \ref{lem2.1} imply the claim.\\
Moreover, by the Sobolev inequality $\mathcal{D}^{s,2}(\R^{N})\subset L^{2^{*}_{s}}(\R^{N})$ (see Theorem \ref{sobhs}), we get the continuous embedding:
\begin{align}\label{cemb}
\mathcal{D}^{s,2}(\R^{N}) \subset L^{p}+L^{q}. 
\end{align}
\end{remark}

\noindent
At this point we are ready to prove the following result
\begin{thm}\label{compactthm}
Let $N\geq 2$, $s\in (0, 1)$ and $2<p<2^{*}_{s}<q$.
Let us denote by $\tilde{\mathcal{D}}^{s,2}_{rad}(\R^{N})$ the space of nonnegative radial decreasing functions in $\mathcal{D}^{s,2}(\R^{N})$.
Then $\tilde{\mathcal{D}}^{s,2}_{rad}(\R^{N})$ is compactly embedded in $L^{p}+L^{q}$.
\end{thm}
\begin{proof}
We proceed as in the proof of  Lemma $3$ in \cite{BF}.\\
Let $\{u_{j}\}$ be a sequence in $\tilde{\mathcal{D}}^{s,2}_{rad}(\R^{N})$ such that as $j\rightarrow +\infty$
$$
u_{j} \rightharpoonup 0 \mbox{ in } \mathcal{D}^{s,2}(\R^{N}).
$$
By (\ref{CO}) follows that there exists a positive constant $C=C(s,N)$ such that
\begin{equation}\label{73}
|u_{j}(x)|\leq C |x|^{-\frac{N-2s}{2}} \mbox{ for any } j\in \N, x\in \R^{N}\setminus \{0\}.
\end{equation}
Fix $\varepsilon>0$. By using (\ref{73}) and $q>2^{\s}$, for $R>0$ big enough we get
\begin{equation}\label{74}
\int_{\{|x|\geq R\}} |u_{j}(x)|^{q} dx\leq C^{q} \int_{\{|x|\geq R\}} \frac{1}{|x|^{q(\frac{N-2s}{2})}} dx< \frac{\varepsilon}{2}
\end{equation}
for all $j\in \N$.
Now, we observe that $\mathcal{D}^{s,2}(\R^{N})\subset H^{s}_{loc}(\R^{N})\Subset L^{p}_{loc}(\R^{N})$
since $p\in (2, 2^{*}_{s})$.
In particular we have
\begin{equation}\label{75}
\int_{\{|x|< R\}} |u_{j}(x)|^{p} dx< \frac{\varepsilon}{2}.
\end{equation}
Taking into account (\ref{74}) and (\ref{75}) we have for $j$ large
\begin{equation}\label{76}
\int_{\{|x|\geq R\}} |u_{j}(x)|^{q} dx+\int_{\{|x|< R\}} |u_{j}(x)|^{p} dx< \varepsilon.
\end{equation}
If $R$ is sufficiently large, by (\ref{73}) follows that for any $j\in \N$
$$
|u_{j}(x)|\leq 1 \mbox{ for } |x|>R.
$$ 
Then for all $j\in \N$
$$
\Gamma_{j}:=\{x\in \R^{N}: |u_{j}(x)|>1\}\subset B_{R}.
$$
Let us observe that
\begin{align}\label{77}
\|u_{j}\|^{q}_{L^{q}(\Gamma_{j}^{c})}+\|u_{j}\|^{p}_{L^{p}(\Gamma_{j})}=\|u_{j}\|^{q}_{L^{q}(B_{R}^{c})}+\|u_{j}\|^{q}_{L^{q}(B_{R}\setminus \Gamma_{j})}+\|u_{j}\|^{p}_{L^{p}(B_{R})}-\|u_{j}\|^{p}_{L^{p}(B_{R}\setminus \Gamma_{j})}.
\end{align}
Here $A^{c}=\R^{N}\setminus A$ for $A\subset \R^{N}$.
Since $p<q$ and $|u_{j}(x)|\leq 1$ in $B_{R}\setminus \Gamma_{j}$ we obtain
\begin{align}\label{78}
\|u_{j}\|^{q}_{L^{q}(B_{R}\setminus \Gamma_{j})}\leq \|u_{j}\|^{p}_{L^{p}(B_{R}\setminus \Gamma_{j})}.
\end{align}
Putting together (\ref{76}), (\ref{77}) and (\ref{78}) we have for $j$ large enough
$$
\|u_{j}\|^{q}_{L^{q}(\Gamma_{j}^{c})}+\|u_{j}\|^{p}_{L^{p}(\Gamma_{j})}\leq \|u_{j}\|^{q}_{L^{q}(B_{R}^{c})}+\|u_{j}\|^{p}_{L^{p}(B_{R})}<\varepsilon
$$
so, in particular, 
\begin{equation}\label{80}
\|u_{j}\|_{L^{q}(\Gamma_{j}^{c})}<\varepsilon^{1/q} \mbox{ and } \|u_{j}\|_{L^{p}(\Gamma_{j})}<\varepsilon^{1/p}.
\end{equation}
Then, by Lemma $1$ and (\ref{80}), we can infer that for $j$ large
$$
\|u_{j}\|_{L^{p}+L^{q}}\leq \max\{\|u_{j}\|_{L^{q}(\Gamma_{j}^{c})},  \|u_{j}\|_{L^{p}(\Gamma_{j})} \} <\max\{\varepsilon^{1/q}, \varepsilon^{1/p} \}.
$$
\end{proof}

\section{proof of Theorem $1$}
\noindent
This section is devoted to the proof of the main result of this paper.\\
In order to obtain a solution to (\ref{P}), we will look critical points of the following functional 
$$
I(u):=\frac{1}{2}[u]_{H^{s}(\R^{N})}^{2}-\int_{\R^{N}} g(u(x)) dx
$$
constrained on
$$
\mathcal{N}=\left \{u\in  \mathcal{D}^{s,2}(\R^{N})\setminus\{0\}: J(u)=0  \right \}
$$
where 
$$
J(u):=\langle I'(u), u\rangle=[u]^{2}-\int_{\R^{N}} g'(u)u dx.
$$
By using the results in Section $2$, we can see that $I$ is well defined on $\mathcal{D}^{s,2}(\R^{N})$ and $I$ is a $C^{2}$-functional.

\begin{proof}(Proof of Theorem \ref{thm1})
We divide the proof in several steps.\\
{\bf Step 1}: $\mathcal{N}$ is a $C^{1}$-manifold. \\
By using $(g1)$ we have, for any $u\in \mathcal{N}$
\begin{align*}
2[u]^{2}-\int_{\R^{N}} (g'(u)u+g''(u)u^{2}) \,dx&=[u]^{2}-\int_{\R^{N}} g''(u)u^{2} \,dx \\
&=\int_{\R^{N}} (g'(u)u-g''(u)u^{2}) \,dx<0. 
\end{align*}

\noindent
{\bf Step 2}: Given $u \neq 0$, there exists a unique $t=t(u)>0$ such that $u t(u)\in \mathcal{N}$
and $I(ut(u))$ is the maximum for $I(tu)$ for $t\geq 0$.\\
Fix $u\neq 0$ and let 
$$
h(t):=I(tu)=\frac{t^{2}}{2}[u]^{2}-\int_{\R^{N}} g(tu(x)) dx
$$
for $t\geq 0$.\\
Then
$$
h'(t)=t[u]^{2}-\int_{\R^{N}} g'(tu(x))u \,dx
$$
and
$$
h''(t)=[u]^{2}-\int_{\R^{N}} g''(tu(x))u^{2} \,dx. 
$$
Let us observe that $t=0$ is a minimum for $h$ since $0=h(0)=h'(0)$ and $h''(0)>0$. Moreover, if $t_{0}>0$ is a critical point of $h$, then by $(g1)$, we obtain that $t_{0}$ is a maximum for $h$ because of 
$$
h''(t_{0})=[u]^{2}-\int_{\R^{N}} g''(t_{0} u(x))u^{2} dx=\int_{\R^{N}} \left(\frac{g'(t_{0}u)}{t_{0}}u-g''(t_{0}u)u^{2}\right) dx<0.
$$ 
By using $(g2)$, we get
\begin{align*}
h(t)&\leq \frac{t^{2}}{2}[u]^{2}-c_{0}t^{q}\int_{t|u|<1} |u|^{q} dx-c_{0}t^{p}\int_{t|u|>1} |u|^{p} dx\\
&\leq \frac{t^{2}}{2}[u]^{2}-c_{0}t^{p}\int_{t|u|>1} |u|^{p} dx\rightarrow -\infty  \quad \mbox{ as } t\rightarrow +\infty
\end{align*}
since $p>2$.
\noindent

{\bf Step 3}: The dependence of $t(u)$ on $u$ is of class $C^{1}$.\\
Let us define the following operator
$$
L(t,u):=t[u]^{2}-\int_{\R^{N}} g'(tu(x))u(x) dx
$$
for $(t,u)\in \R_{+} \times \mathcal{D}^{s,2}(\R^{N})$.
By Lemma \ref{lem2.1} we can see that $L\in C^{1}$ and if $(t_{0},u_{0})$ is a point such that $L(t_{0},u_{0})=0$ and $t_{0}, u_{0}\neq 0$, then by $(g_1)$
\begin{align*}
\frac{d}{dt}L(t_{0},u_{0})&=[u]^{2}-\int_{\R^{N}} g''(t_{0}u_{0}(x))u^{2}_{0} dx\\
&=\int_{\R^{N}} \left(\frac{g'(t_{0}u_{0})}{t_{0}}u_{0}-g''(t_{0}u_{0})u_{0}^{2}\right) \, dx<0.
\end{align*}
By invoking the Implicit Function Theorem we obtain that $u\mapsto t(u)$ is $C^{1}$ and
$$
\langle t'(u_{0}), v\rangle=\frac{\displaystyle{t_{0}^{2} \int_{\R^{N}} 2t_{0} (-\Delta)^{\frac{s}{2}}u_{0} (-\Delta)^{\frac{s}{2}}v- g'(t_{0}u_{0})v-g''(t_{0}u_{0})t_{0}u_{0}v \, \, dx }}{\displaystyle{\int_{\R^{N}} g''(t_{0}u_{0})t^{2}_{0}u^{2}_{0}-g'(t_{0}u_{0}) t_{0}u_{0}\, \,  dx}}
$$
where $t_{0}=t(u_{0})$.

\noindent
{\bf Step 4}: $\displaystyle{\inf_{v\in \mathcal{N}} [v]^{2}>0}$.\\
Let $v_{j}$ be a minimizing sequence in $\mathcal{N}$. We assume by contradiction that $v_{j}$ converges to zero in $\mathcal{D}^{s,2}(\R^{N})$. 
We set $t_{j}=[v_{j}]$, hence we can write $v_{j}=t_{j} u_{j}$ where $[u_{j}]=1$. 
Since the embedding $\mathcal{D}^{s,2}(\R^{N})\subset L^{p}+L^{q}$ is continuous, we deduce that $u_{j}$ is bounded in $L^{p}+L^{q}$. 
Then, by $(v_{j})\subset \mathcal{N}$, $t_{j}\rightarrow 0$ and Remark \ref{rem1}, we get
\begin{align*}
t_{j} &=\frac{1}{t_{j}}[v_{j}]^{2}=\int_{\R^{N}} g'(t_{j}u_{j})u_{j} dx\\
&\leq c_{1}t_{j}^{q-1}\int_{\{|v_{j}|\leq 1\}} |u_{j}|^{q} dx+c_{1}t_{j}^{p-1}\int_{\{|v_{j}|>1\}} |u_{j}|^{p} dx\\
&\leq c_{1}t_{j}^{q-1}\int_{\{|v_{j}|\leq 1\}} |u_{j}|^{q} dx+c_{1}t_{j}^{p-1}\int_{\{|u_{j}|>1\}} |u_{j}|^{p} dx \\
&\leq c_{1}t_{j}^{q-1}\int_{\{|u_{j}|\leq 1\}} |u_{j}|^{q} dx+c_{1}t_{j}^{q-1}\int_{\{|v_{j}|\leq 1\}\cap \{|u_{j}|> 1\}} \frac{|v_{j}|^{p}}{t_{j}^{q-p}} dx \\
&+c_{1}t_{j}^{p-1}\int_{\{|u_{j}|>1\}} |u_{j}|^{p} dx \\
&\leq c_{1}t_{j}^{q-1}\int_{\{|u_{j}|\leq 1\}} |u_{j}|^{q} dx+2c_{1}t_{j}^{p-1}\int_{\{|u_{j}|>1\}} |u_{j}|^{p} dx
\end{align*}
that is
\begin{equation}\label{tj}
1\leq c_{1}t_{j}^{q-2}\int_{\{|u_{j}|\leq 1\}} |u_{j}|^{q} dx+2c_{1}t_{j}^{p-2}\int_{\{|u_{j}|>1\}} |u_{j}|^{p} dx. 
\end{equation}
Taking into account (\ref{tj}), $(b)$ of Lemma \ref{lem2.1} and $t_{j}\rightarrow 0$, we have
$$
1\leq c'_{1} t_{j}^{q-2}+c'_{2}t_{j}^{p-2} \rightarrow 0 \mbox{ as } j \rightarrow \infty,
$$
that is a contradiction. Therefore $\inf_{v\in \mathcal{N}} [v]^{2}>0$.

\medskip
\noindent
Since we are looking for positive solutions to (\ref{P}), we can assume that $g(t)=0$ for $t \leq 0$.\\
{\bf Step 5}: The following infimum
\begin{equation}\label{inf}
m:=\inf_{v\in \mathcal{N}} I(v)
\end{equation}
is achieved.\\
Let $\{u_{j}\}\subset \mathcal{D}^{s,2}(\R^{N})$ be a minimizing sequence for (\ref{inf}).
Then, by $(g1)$, follows that
\begin{align}\label{3.3}
\Bigl(\frac{1}{2}-\frac{1}{\mu}\Bigr)[u_{j}]^{2}\leq \frac{1}{2}[u_{j}]^{2}-\int_{\R^{N}} g(u_{j}) dx=I(u_{j}),
\end{align}
that is $\{u_{j}\}$ is bounded in $\mathcal{D}^{s,2}(\R^{N})$.\\
We claim that $m>0$. Indeed, if $m=0$, the minimizing sequence $\{u_{j}\} \subset \mathcal{N}$ is such that $I(u_{j}) \rightarrow 0$, and by (\ref{3.3}) we deduce that $[u_{j}] \rightarrow 0$. This gives a contradiction because of  Step $4$.\\
Now, by using Theorem \ref{park}, we can note that $I(u^{*}_{j})\leq I(u_{j})$, where $u^{*}_{j}$ is the symmetric-decreasing rearrangement  of $u_{j}$. Moreover, by the boundedness of $\{u_{j}\}$, we can see that $\{u^{*}_{j}\}$ is bounded in $\mathcal{D}^{s,2}(\R^{N})$.\\
In virtue of Theorem \ref{compactthm}, the embedding $\tilde{\mathcal{D}}^{s,2}_{rad}(\R^{N})\subset L^{p}+L^{q}$ is compact, so, up to a subsequence, we may assume that $u^{*}_{j} \rightarrow u^{*}$ strongly in $L^{p}+L^{q}$, and weakly in $\mathcal{D}^{s,2}(\R^{N})$.\\
Let us observe that 
$
J(u^{*}_{j}) \leq J(u_{j})= 0
$
thanks to Theorem \ref{park},
so we don't know if $u_{j}^{*}$ belongs to the Nehari manifold $\mathcal{N}$. 
Then, for any $j\in \N$ there exists a unique $t_{j}\in [0,1]$ such that $t_{j} u^{*}_{j}\in \mathcal{N}$ and $t_{j}$ converges to some $t_{0}$.
By Step $3$ follows that $I(u_{j})$ is the maximum for $I(tu_{j})$ when $t \geq 0$, so we get
\begin{equation}\label{3.5}
0<m\leq I(t_{j} u^{*}_{j})\leq I(t_{j} u_{j})\leq I(u_{j}). 
\end{equation}
Since $I(u_{j})\rightarrow m$, we obtain $I(t_{j} u^{*}_{j}) \rightarrow m$. It is clear that $t_{0} \neq 0$. Otherwise $t_{j} u^{*}_{j}\rightarrow 0$ in $\mathcal{D}^{s,2}(\R^{N})$, and by (\ref{3.5}) we deduce that $m=0$, which provides a contradiction in virtue of Step $4$.\\
Now, we show that $u^{*}\neq 0$. If we suppose that $u^{*}=0$, then $u^{*}_{j} \rightarrow 0$ strongly in $L^{p}+L^{q}$.
Putting together $J(u^{*}_{j})\leq 0$, Remark \ref{rem1}, and by using H\"older inequality and Lemma \ref{lem2.1}, follows that by setting $u^{*}_{j}=u^{*}_{1j}+u^{*}_{2j}$ with $u^{*}_{1j}\in L^{p}$ and $u^{*}_{2j}\in L^{q}$ and $\Gamma_{j}=\{x\in \R^{N}: |u_{j}(x)|>1\}$
\begin{align*}
[u^{*}_{j}]^{2}&\leq \int_{\R^{N}} g'(u_{j}^{*})u_{1j}^{*} dx+\int_{\R^{N}} g'(u_{j}^{*})u_{2j}^{*} dx \\
&\leq \left[c_{1} \|u_{j}^{*}\|^{q/p'}_{L^{q}(\R^{N}-\Gamma_{j})}+c_{1} \|u_{j}^{*}\|^{p-1}_{L^{p}(\Gamma_{j})}\right] \|u_{1j}^{*}\|_{L^{p}(\R^{N})} \\
&+ \left[c_{1} \|u_{j}^{*}\|^{q-1}_{L^{q}(\R^{N}-\Gamma_{j})}+c_{1}|\Gamma_{j}|^{\frac{p-q}{pq}} \|u_{j}^{*}\|^{p-1}_{L^{p}(\Gamma_{j})}\right] \|u_{2j}^{*}\|_{L^{q}(\R^{N})}\\
&\leq c_{4} \|u_{j}^{*}\|_{L^{p}+L^{q}}
\end{align*}
that is $u^{*}_{j} \rightarrow 0$ in $\mathcal{D}^{s,2}(\R^{N})$ as $j \rightarrow \infty$. This and (\ref{3.5}) yield $m=0$, which is impossible because $m>0$.\\
Then, by using $I(t_{j} u^{*}_{j}) \rightarrow m$ and $t_{j}u^{*}_{j}\in \mathcal{N}$, and by applying Lemma \ref{lem2.3} and Lemma \ref{lem2.5}, we have
\begin{align}\label{cane}
m &=\lim_{j \rightarrow \infty} I(t_{j}u^{*}_{j}) \nonumber\\
&=\lim_{j \rightarrow \infty}  \int_{\R^{N}} \frac{1}{2} g'(t_{j}u^{*}_{j})t_{j}u^{*}_{j}-g(t_{j}u^{*}_{j}) dx \nonumber\\
&= \int_{\R^{N}} \frac{1}{2} g'(t_{0}u^{*})t_{0}u^{*}-g(t_{0}u^{*}) dx.
\end{align}
Now, we argue by contradiction in order to prove that $t_{0}u^{*} \in \mathcal{N}$. If we 
assume that $t_{0}u^{*} \notin \mathcal{N}$, by $J(u^{*}_{j}) \leq J(u_{j})= 0$ and (\ref{3.5}) we have 
$$
I(t_{0}u^{*})\leq m
$$
and
$$
[t_{0}u^{*}]^{2}<\int_{\R^{N}} g'(t_{0}u^{*}) t_{0}u^{*} \, dx, 
$$
so we can find $t_{1}\in [0,1)$ such that $t_{1}t_{0}u^{*} \in \mathcal{N}$.
As a consequence
$$
m\leq I(t_{1}t_{0}u^{*})=\int_{\R^{N}} \frac{1}{2} g'(t_{1}t_{0}u^{*}) t_{1}t_{0}u^{*}-g(t_{1}t_{0}u^{*}) \,dx.
$$ 
In view of $(g1)$, the map
$$
t>0\mapsto \psi(t):=\int_{\R^{N}} \frac{1}{2} g'(tu) tu-g(tu) \, dx
$$
is strictly increasing, so by this and (\ref{cane}) we get
$$
m\leq \psi(t_{1}t_{0}u^{*})<\psi(t_{0}u^{*})= m,
$$
which is a contradiction.
This concludes the proof of Step $5$.\\
Then, by applying the Lagrange multiplier rule, there exists $\lambda\in \R$ such that
\begin{equation}\label{WF}
\langle I'(t_{0} u^{*}), \varphi \rangle=\lambda \langle J'(t_{0} u^{*}), \varphi \rangle
\end{equation}
for any $\varphi \in \mathcal{D}^{s, 2}(\R^{N})$.\\
By testing $\varphi=t_{0} u^{*}\in \mathcal{N}$ in (\ref{WF}), and keeping in mind that $\langle J'(u), u \rangle <0$ for all $u\in \mathcal{N}$ (see Step $1$), we deduce that $0=\langle I'(t_{0} u^{*}), t_{0} u^{*} \rangle=\lambda \langle J'(t_{0} u^{*}), t_{0} u^{*}\rangle$, that is $\lambda=0$ and $t_{0} u^{*}$ is a nontrivial solution to (\ref{P}).
Actually, by the strong maximum principle \cite{CS1}, we argue that $t_{0} u^{*}$ is positive.
\end{proof}

\end{document}